\documentclass[12pt,a4paper,english,reqno]{amsart}
\usepackage[a4paper,margin=1.1in,footskip=3em]{geometry}
\usepackage{amsmath,amssymb,amsthm,mathtools}
\usepackage[mathscr]{euscript}
\usepackage[usenames,dvipsnames]{color}
\usepackage{adjustbox,tikz,calc,graphics,babel,standalone}
\usepackage{subcaption}
\usepackage{csquotes,enumerate,verbatim}
\usepackage[final]{microtype}
\usepackage[numbers]{natbib}
\usetikzlibrary{shapes.misc,calc,intersections,patterns,decorations.pathreplacing}
\usepackage{hyperref}
\hypersetup{colorlinks=true,linkcolor=blue,citecolor=blue}

\theoremstyle{plain}
\newtheorem{theorem}{Theorem}

\newtheorem{proposition}[theorem]{Proposition}

\newtheorem{conjecture}[theorem]{Conjecture}
\newtheorem*{theorem*}{Theorem}
\newtheorem*{lemma*}{Lemma}

\def\N{\mathbb{N}}

\def\Alphabet{\mathscr{A}}
\def\C{\mathcal}

\let\eps\varepsilon

\let\originalleft\left
\let\originalright\right
\renewcommand{\left}{\mathopen{}\mathclose\bgroup\originalleft}
\renewcommand{\right}{\aftergroup\egroup\originalright}

\makeatletter
\newcommand*\bigcdot{\mathpalette\bigcdot@{.6}}
\newcommand*\bigcdot@[2]{\mathbin{\vcenter{\hbox{\scalebox{#2}{$\m@th#1\bullet$}}}}}
\makeatother

\makeatletter
\def\imod#1{\allowbreak\mkern10mu({\operator@font mod}\,\,#1)}
\makeatother

\begin{document}

\title{Product-free sets in the free semigroup}

\author{Imre Leader}
\address{Department of Pure Mathematics and Mathematical Statistics, University of Cambridge, Wilberforce Road, Cambridge CB3\thinspace0WB, UK}
\email{i.leader@dpmms.cam.ac.uk}

\author{Shoham Letzter}
\address{ETH Institute for Theoretical Studies, 8092 Zurich, Switzerland}
\email{shoham.letzter@math.ethz.ch}

\author{Bhargav Narayanan}
\address{Department of Mathematics,	Rutgers University, Piscataway NJ 08854, USA}
\email{narayanan@math.rutgers.edu}

\author{Mark Walters}
\address{School of Mathematical Sciences, Queen Mary, University of London, London E1\thinspace4NS, UK}
\email{m.walters@qmul.ac.uk}

\date{6 December 2018}

\begin{abstract}
	In this paper, we study product-free subsets of the free semigroup over a finite alphabet $\Alphabet$. We prove that the maximum density of a product-free subset of the free semigroup over $\Alphabet$, with respect to the natural measure that assigns a weight of $|\Alphabet|^{-n}$ to each word of length $n$, is precisely $1/2$.
\end{abstract}

\maketitle
\section {Introduction}
A subset $S$ of a semigroup is said to be \emph{product-free} if there do not exist $x,y,z \in S$ (not necessarily distinct) such that $x\bigcdot y = z$; it is customary to call $S$ \emph{sum-free} when the underlying semigroup is abelian.

It is a well known fact (and an easy exercise) that any sum-free subset of the integers has upper density at most $1/2$. Sum-free subsets of the integers, and of abelian groups in general, have been studied by very many researchers over the last fifty years. For example, from the work of Green and Ruzsa~\citep{green}, there is now a complete picture of how large a sum-free set we can find in any finite abelian group. We refer the reader to the surveys of Tao and Vu~\citep{tao-s} and Kedlaya~\citep{ked-s} for more information on these questions.

Product-free subsets of finite non-abelian groups were first investigated by Babai and S\'os~\citep{babai}. Following foundational work by Gowers~\citep{gowers} demonstrating so-called `product-mixing' phenomena in groups with no low-dimensional representations, there has been a great deal of recent work in the non-abelian setting; for instance, in a recent breakthrough, Eberhard~\citep{eber} determined how large a product-free subset of the alternating group can be.

In light of these developments, it is natural to ask what one can say about product-free sets in infinite non-abelian structures, a setting in which our knowledge is a bit more limited. Perhaps the first natural place to look among infinite non-abelian structures is among those that are free, so here, we shall investigate how large product-free subsets of the free semigroup can be.

\section{Our results}
Let $\Alphabet$ be a finite set. We write $\C{F} = \C{F}_\Alphabet$ for the free semigroup over $\Alphabet$; in other words, $\C{F}$ is the set of all finite words over the alphabet $\Alphabet$ equipped with the associative operation of concatenation. While we state and prove our results for finite alphabets of all possible sizes for the sake of completeness, the reader will lose nothing by supposing that $\Alphabet$ is a two-element set in what follows; indeed, this case captures all the difficulties inherent in the questions we study.

Recall that a set $S \subset \C{F}$ is \emph{product-free} if, writing $\bigcdot$ for the operation of concatenation, there do not exist words $x,y,z \in S$ (not necessarily distinct) such that $x\bigcdot y = z$. There is an obvious example of a `large' subset of $\C{F}$ that is product-free: when $\Alphabet = \{ a, b\}$ for instance, the set of words which contain an odd number of occurrences of the symbol $a$ (or $b$, for that matter) is easily seen to be a product-free set that contains, roughly, half the words from $\C{F}$.
Our aim in this paper is to prove that these sets are, in a precise sense, the largest product-free subsets of $\C{F}$. We remark in passing that there are several other product-free sets that are `equally large': for any nonempty subset $\Gamma \subset \Alphabet$, the \emph{odd-occurrence set} $\C{O}_\Gamma \subset \C{F}$ generated by $\Gamma$, namely the set of words in which the total number of occurrences of symbols from $\Gamma$ is odd, is easily seen to be a product-free set; in the case where $\Alphabet = \{ a, b\}$, our earlier example corresponds to taking $\Gamma = \{a\}$, and taking $\Gamma = \{a, b\}$ gives us the set of all words of odd length, for example.

To formally state our results, we need a way to measure the size of a set $S\subset \C{F}$. For an integer $n\in \N$, the \emph{layer} $\C{F}(n) \subset \C{F}$ is the set of words of length $n$, and the \emph{ball} $\C{F}_{\le}(n) \subset \C{F}$ is the set of words of length at most $n$. As a first attempt, one might define the density of a set $S \subset \C{F}$ via its densities in balls, namely as the quantity
\[\limsup_{n \to \infty}\frac{|S \cap \C{F}_{\le}(n)|}{|\C{F}_{\le}(n)|}.\]
However, a little thought should convince the reader that the counting measure is somewhat ill-suited for our purposes. Indeed, when $|\Alphabet| > 1$, almost all the words in $\C{F}_{\le}(n)$ are long since $|\C{F}(n)| \ge |\C{F}_{\le}(n)|/2$. Consequently, we may find product-free sets that are intuitively small, and yet have density arbitrarily close to $1$ in the above sense; for example, for any sufficiently large $c \in \N$, the set
\[\bigcup_{n\ge c}(\C{F}_{\le 2^n + c} \setminus \C{F}_{\le 2^n})\]
is product-free and has density at least $1-1/c$ in the above sense, provided $|\Alphabet| > 1$.

A more natural approach is to assign a weight of $|\Alphabet|^{-n}$ to each word of $\C{F}(n)$, thereby ensuring that the layers $\C{F}(n)$ have the same total weight for all $n\in\N$. To this end, for a subset $S \subset \C{F}$ and an integer $n \in \N$, we define \emph{the density of $S$ in the layer $\C{F}(n)$} by $d_S(n) = |S \cap \C{F}(n)|/|\C{F}(n)|$. With this definition in place, most standard notions of density may now be carried over: we define the \emph{upper asymptotic density} of $S$ by
\[ \bar d(S) = \limsup_{n \to \infty} \frac{ \sum_{i=1}^{n}d_S(i)}{n},\]
and the \emph{upper Banach density} of $S$ by
\[ d^*(S) = \limsup_{n-m \to \infty} \frac{ \sum_{i=m}^{n}d_S(i)}{n-m+1}.\]
Of course, the latter is a weaker notion of density than the former; indeed, it is clear that $\bar d(S) \le d^*(S)$ for any $S \subset \C{F}$.

It is easy to see that any odd-occurrence set has both an upper asymptotic density and an upper Banach density of $1/2$. Our aim in this note is to show that product-free sets cannot be any larger; our main result is as follows.
\begin{theorem}\label{main-res}
	Let $\Alphabet$ be a finite set. If $S \subset \C{F}_\Alphabet$ is product-free, then $d^*(S) \le 1/2$.
\end{theorem}

Let us mention that product-free sets in cancellative semigroups have been studied by {\L}uczak and Schoen~\citep{semigr}; while their results are sharp for such semigroups in general, these results do not give us any effective bounds on the size of a product-free subset of $\C{F}$.

Before we turn to the proof of Theorem~\ref{main-res}, it is worth pointing out that there is a simple argument that allows us to bound the upper asymptotic density of a product-free subset of $\C{F}$ away from $1$. Indeed, suppose that $S\subset\C{F}$ is product-free. We then have
\[d_S(m)d_S(n) + d_S(m+n) \le 1\]
for any $m,n\in\N$ since the sets $S \cap \C{F}(m+n)$ and $(S \cap \C{F}(m))\bigcdot (S \cap \C{F}(n))$ must be disjoint. Now, consider the set of integers $n \in \N$ for which $d_S(n) > \phi$, where $\phi = (\sqrt 5 - 1)/2 \approx 0.618$ is the unique positive solution to the equation $x^2 + x = 1$. It follows from the inequality above that this set of integers must be sum-free. It is now easy to see that $\bar d(S) \le (1 + \phi)/2 \approx 0.809$.

We shall have to work somewhat harder to prove Theorem~\ref{main-res}, which improves this bound of $(1 + \phi)/2$ for the upper asymptotic density to the optimal bound of $1/2$ for the upper Banach density. The proof of Theorem~\ref{main-res} is given in Section~\ref{sec-proof}. We conclude this note with a discussion of some open problems in Section~\ref{sec-conc}.

\section{Proof of the main result}\label{sec-proof}

We begin by fixing our finite alphabet $\Alphabet$. In the sequel, $\C{F}$ will always mean $\C{F}_\Alphabet$, the free semigroup over this fixed alphabet $\Alphabet$.

It will be helpful to establish some notation. For a pair of words $x, w \in \C{F}$, we say that $x$ is a \emph{prefix} of $w$ if $w = x\bigcdot y$ for some $y \in \C{F}$, and that $x$ is a \emph{suffix} of $w$ if $w = y\bigcdot x$ for some $y \in \C{F}$. For a pair of sets $S_1, S_2 \subset \C{F}$, we write $S_1 \bigcdot S_2$ for  their (Minkowski) product; in other words,
\[S_1 \bigcdot S_2 = \{ w_1 \bigcdot w_2 : w_1 \in S_1, w_2 \in S_2\}.\]

For a set $S\subset \C{F}$ and an integer $n \in \N$, we set $S(n) = S \cap \C{F}(n)$. One of the key ideas in the proof of Theorem~\ref{main-res} is the following definition. For any sequence of positive integers $ \ell_1 < \ell_2 <\dots <\ell_k <n$, we define
\[ S(n; \ell_1, \ell_2,\dots,\ell_k) = \left\{ w \in S(n):  w \text{ has no prefix in } S(\ell_1) \cup S(\ell_2) \cup \dots \cup  S(\ell_k) \right \}; \]
in other words,
\[ S(n; \ell_1, \ell_2,\dots,\ell_k) = S(n) \setminus \left( \bigcup_{i=1}^{k}S(\ell_i) \bigcdot \C{F}(n-\ell_i) \right).
\]
Let us note, for any $S \subset \C{F}$, that the sets $S(n;m)$ and $S(m) \bigcdot \C{F}(n-m)$ are disjoint for any pair of positive integers $m < n$. Recall that $d_S(n) = |S(n)||\C{F}(n)|^{-1}$; we analogously define
\[d_S(n; \ell_1, \ell_2,\dots,\ell_k) = \frac{|S(n; \ell_1, \ell_2,\dots,\ell_k)|}{|\C{F}(n)|}.\]

When the set $S$ in question is clear, we write $d(n)$ and $d(n; \ell_1, \ell_2,\dots,\ell_k)$ for $d_S(n)$ and $d_S(n; \ell_1, \ell_2,\dots,\ell_k)$, respectively. Recall that for any product-free set $S \subset \C{F}$ and any $m,n \in \N$, we have
\[d(m)d(n) + d(m+n) \le 1.\]
We start by proving a generalisation of this fact.
\begin{proposition}\label{doublecount}
	If $S \subset \C{F}$ is product-free, then for any sequence of positive integers $\ell_1 < \ell_2 <\dots <\ell_k <n$, we have
	\begin{align*}
		         & d(\ell_1)d(n-\ell_1) + d(\ell_2;\ell_1)d(n-\ell_2) + \dots + d(\ell_k; \ell_1,\ell_2,\dots,\ell_{k-1})d(n-\ell_k) + d(n)      \\
		\le \,\, & d(\ell_1) + d(\ell_2;\ell_1) + \dots + d(\ell_k; \ell_1,\ell_2,\dots,\ell_{k-1}) + d(n; \ell_1, \ell_2, \dots, \ell_k) \le 1.
	\end{align*}
\end{proposition}
\begin{proof}
	First, consider the products
	\[ S(\ell_1) \bigcdot S(n-\ell_1), S(\ell_2;\ell_1)\bigcdot S(n-\ell_2), \dots, S(\ell_k; \ell_1,\ell_2,\dots,\ell_{k-1}) \bigcdot S(n-\ell_k).\]
	These subsets of $\C{F}(n)$ are by definition disjoint. Let $L'$ be the union of these $k$ sets. Since $S$ is product-free, $L'$ and $S(n)$ are disjoint as well. Let $L = L' \cup S(n)$; clearly, the density of $L$ in $\C{F}(n)$ is
	\[d(\ell_1)d(n-\ell_1) + d(\ell_2;\ell_1)d(n-\ell_2) + \dots + d(\ell_k; \ell_1,\ell_2,\dots,\ell_{k-1})d(n-\ell_k) + d(n).
	\]

	Next, consider the Minkowski products
	\[
		S(\ell_1) \bigcdot \C{F}(n-\ell_1),\, S(\ell_2;\ell_1)\bigcdot \C{F}(n-\ell_2),\, \dots,\, S(\ell_k; \ell_1,\ell_2,\dots,\ell_{k-1}) \bigcdot \C{F}(n-\ell_k).
	\]
	These subsets of $\C{F}(n)$ are again disjoint by definition; let $R'$ denote their union. Note that $R'$ and $S(n; \ell_1, \ell_2, \dots, \ell_k)$ are disjoint. Let $R = R' \cup  S(n; \ell_1, \ell_2, \dots, \ell_k)$; it is easy to see that the density of $R$ in $\C{F}(n)$ is
	\[d(\ell_1) + d(\ell_2;\ell_1) + \dots + d(\ell_k; \ell_1,\ell_2,\dots,\ell_{k-1}) + d(n; \ell_1, \ell_2, \dots, \ell_k)\]
	and that this quantity is therefore at most $1$.

	To finish the proof, it suffices to show that
	\[L' \cup S(n) = L \subset R = R' \cup S(n; \ell_1, \ell_2, \dots, \ell_k).\]
	It is easy to see that $L' \subset R'$. Therefore, it is sufficient to show that $S(n)$ is a subset of $R' \cup S(n; \ell_1, \ell_2, \dots, \ell_k)$. To see this, note that any word from $S(n)$ which has a prefix in $S(\ell_1) \cup S(\ell_2) \cup \dots \cup S(\ell_k)$ is also contained in $R'$. In other words, $S(n) \setminus S(n; \ell_1, \ell_2, \dots, \ell_k) \subset R'$; the result follows.
\end{proof}

With the above observation in hand, we are now ready to prove Theorem~\ref{main-res}.
\begin{proof}[Proof of Theorem~\ref{main-res}]
	We prove by contradiction that the upper Banach density of a product-free set is at most $1/2$.

	Suppose that $S \subset \C{F}$ is product-free and that $d^*(S) > 1/2 + \eps$ for some $\eps > 0$. We then claim that we may find an increasing sequence of positive integers $(\ell_k)_{k \in \N}$ such that
	\[ d(\ell_1) + d(\ell_2; \ell_1) + \dots + d(\ell_k; \ell_1, \ell_2, \dots, \ell_{k-1}) \ge \frac{1}{2} + \frac{1}{4} + \dots + \frac{1}{2^k} = 1 - \frac{1}{2^k}\]
	for each $k \in \N$.

	We construct this sequence inductively. Since $d^*(S) > 1/2$, it is clear that we may find $\ell_1 \in \N$ such that $d(\ell_1) \ge 1/2$. Having found $\ell_1< \ell_2< \dots< \ell_k$ as required, we choose $\ell_{k+1}$ as follows. Since $d^*(S) > 1/2 + \eps$, there exist arbitrarily long intervals $I \subset \N$ that satisfy
	\[
		\frac{\sum_{n \in I}d(n)}{|I|} > \frac{1}{2} + \eps.
	\]
	Choose such an interval $I$ whose length is sufficiently larger than $\ell_k$; we may assume, by passing to a sub-interval if necessary, that $\min I > \ell_k$. We claim that it is possible to choose $\ell_{k+1}$ from $I$; in other words, we claim that there exists an $n \in I$ such that
	\begin{multline*}
		d(\ell_1) + d(\ell_2; \ell_1) + \dots + d(\ell_k; \ell_1, \ell_2, \dots, \ell_{k-1}) + d(n; \ell_1, \ell_2, \dots, \ell_{k}) \ge 1 - \frac{1}{2^{k+1}}.
	\end{multline*}

	We prove this claim by contradiction. Suppose that there is no such $n \in I$. Then, by Proposition~\ref{doublecount}, we have
	\begin{align*}
		         & d(\ell_1)d(n-\ell_1) + d(\ell_2;\ell_1)d(n-\ell_2) + \dots + d(\ell_k; \ell_1,\ell_2,\dots,\ell_{k-1})d(n-\ell_k) + d(n)                     \\
		\le \,\, & d(\ell_1) + d(\ell_2;\ell_1) + \dots + d(\ell_k; \ell_1,\ell_2,\dots,\ell_{k-1}) + d(n; \ell_1, \ell_2, \dots, \ell_k) < 1-\frac{1}{2^{k+1}}
	\end{align*}
	for each $n \in I$. By summing the above inequality over all $n \in I$, we get
	\[
		\sum_{n \in I'}d(n)\left(1+d(\ell_1)+  d(\ell_2;\ell_1) + \dots + d(\ell_k; \ell_1,\ell_2,\dots,\ell_{k-1})\right) < |I|\left( 1 -\frac{1}{2^{k+1}}\right),
	\]
	where $I'\subset I$ is the set of $n \in I$ with $n + \ell_k < \max I$. This implies, by the inductive hypothesis, that
	\[\sum_{n \in I'}d(n)\left(2 - \frac{1}{2^k}\right) < |I|\left(1-\frac{1}{2^{k+1}}\right),\]
	or equivalently, $\sum_{n \in I'}d(n) < |I|/2$. Therefore, we have
	\[ \sum_{n \in I}d(n) \le \sum_{n \in I'}d(n) + \ell_k + 1 < \frac{|I|}{2} + \ell_k + 1,\]
	which contradicts the fact that $\sum_{n \in I}d(n) > |I|/2 + \eps|I|$, provided $|I| > (\ell_k+1)/\eps$.

	We now finish the proof of the proposition by showing that the existence of this sequence $(\ell_k)_{k \in \N}$ contradicts our initial assumption that $d^*(S) > 1/2 + \eps$. Fix a $k\in \N$ large enough to ensure that
	\[\frac{2^k}{2^{k+1} - 1} < \frac{1 + \eps}{2}\]
	and consider any interval $I \subset \N$ with $|I| > 4(\ell_k + 1) / \eps$. We know from Proposition~\ref{doublecount} that
	\[ d(\ell_1)d(n-\ell_1) + d(\ell_2;\ell_1)d(n-\ell_2) + \dots + d(\ell_k; \ell_1,\ell_2,\dots,\ell_{k-1})d(n-\ell_k) + d(n) \le 1\]
	for each $n \in \N$ with $n > \ell_k$; summing this inequality over such $n \in I$, we get
	\[\sum_{n \in I'} d(n)\left(2 - \frac{1}{2^k}\right) \le |I|, \]
	where $I'$ is the set of $n \in I$ with $n > \ell_k$ and $n + \ell_k < \max I$. Therefore,
	\[ \frac{\sum_{n \in I} d(n)}{|I|} \le \frac{2^k}{(2^{k+1} - 1)} + \frac{2(\ell_k + 1)}{|I|} < \frac{1}{2} + \eps, \]
	which is a contradiction; this proves the claimed upper bound in Theorem~\ref{main-res}.
\end{proof}

\section{Conclusion}\label{sec-conc}

A common line of enquiry in the study of product-free sets is to ask for `asymmetric' versions of results bounding the upper density of product-free sets. In this spirit, it is natural to ask whether an analogue of Theorem~\ref{main-res} continues to hold when one wishes to solve the equation $x \bigcdot y = z$ with $x$, $y$ and $z$ in specified subsets of $\C{F}$. More precisely, if $X, Y, Z \subset \C{F}$ are such that there are no solutions to $x \bigcdot y = z$ with $x \in X$, $y \in Y$ and $z\in Z$, one might ask if one of $X$, $Y$ or $Z$ has an upper asymptotic density of at most $1/2$. However, it is not hard to construct for any $\eps>0$, three sets $X, Y, Z \subset \C{F}$, each of upper asymptotic density at least $\phi - \eps$, where $\phi = (\sqrt 5 - 1)/2$, such that there are no solutions to $x \bigcdot y = z$ with $x \in X$, $y \in Y$ and $z\in Z$. Indeed, pick a suitably large $n \in\N$ and choose any set $W \subset \C{F}(n)$ such that $||W|/|\C{F}(n)| - \phi| < \eps/3$. Now take $X$ to be the set of all words with a prefix in $W$, $Y$ to be the set of all words with a suffix in $W$, and $Z$ to be the set $\C{F} \setminus (X \bigcdot Y)$. Clearly, there are no solutions to $x \bigcdot y = z$ with $x \in X$, $y \in Y$ and $z\in Z$; it is also not hard to check that each of $X$, $Y$ and $Z$ has an upper asymptotic density at least $\phi - \eps$.

Next, it would be interesting to understand what product-free sets of maximal density look like. As we saw earlier, several non-isomorphic extremal constructions are furnished by the family of odd-occurrence sets. We suspect that these might be the only constructions of maximal density, and make the following conjecture.
\begin{conjecture}\label{unique}
	Let $\Alphabet$ be a finite set. If $S \subset \C{F}_\Alphabet$ is product-free and $d^*(S) = 1/2$, then $S \subset \C{O}_\Gamma$ for some nonempty subset $\Gamma \subset \Alphabet$.
\end{conjecture}

Finally, another natural direction is to study product-free subsets of the \emph{free group $\mathbf{F}_\Alphabet$} over a finite alphabet $\Alphabet$. Similarly to the situation in this paper, the most natural measure to consider in the case of the free group $\mathbf{F}_\Alphabet$ would be the one that assigns a weight of $|\Alphabet|(|\Alphabet|-1)^{-(n-1)}$ to each irreducible word of length $n$. The different notions of density defined here for the free semigroup then have analogous definitions in the free group, and we believe that an analogue of Theorem~\ref{main-res} should hold in the free group as well; concretely, we conjecture the following.

\begin{conjecture}\label{freegrp}
	For any finite alphabet $\Alphabet$, no product-free subset of the free group $\mathbf{F}_\Alphabet$ has upper Banach density exceeding $1/2$.
\end{conjecture}

Note that, in the proof of Theorem~\ref{main-res}, we rely crucially on the fact that there is exactly one way to write a word of length $m+n$ as the concatenation of a word of length $m$ with a word of length $n$; of course, we lose this property when working with free groups, so we believe that some new ideas will be required to understand product-free sets in free groups.

\section*{Acknowledgements}
The second author would like to acknowledge the support of Dr.\ Max R\"ossler, the Walter Haefner Foundation, and the ETH Zurich Foundation.
The third author wishes to acknowledge support from NSF grant DMS-1800521.
\bibliographystyle{amsplain}
\bibliography{prodfree_semigroup}

\end{document}